\documentclass[11pt,twoside,a4paper]{amsart}
\usepackage{amssymb}
\usepackage{amsmath,amscd}

\usepackage[initials,nobysame,shortalphabetic]{amsrefs}
\usepackage[all,cmtip]{xy}
\usepackage{hyperref}

\date{\today}
\usepackage{geometry}
\usepackage{xcolor}

\def\dbar{\bar\partial}

\def\R{{\mathbf R}}
\def\N{{\mathbf N}}

\def\P{{\mathbf P}}

\def\be{\begin{equation}}
\def\ee{\end{equation}}

\def\Ok{\mathcal O}

\def\1{{\bf 1}}

\newcommand{\ph}{\varphi}

\newcommand{\minus}{\setminus}

\definecolor{darkred}{rgb}{.77, .01,.2}

\mathcode`l="8000
\begingroup
\makeatletter
\lccode`\~=`\l
\DeclareMathSymbol{\lsb@l}{\mathalpha}{letters}{`l}
\lowercase{\gdef~{\ifnum\the\mathgroup=\m@ne \ell \else \lsb@l \fi}}%
\makeatother
\endgroup

\newtheorem{thm}{Theorem}[section]
\newtheorem{lma}[thm]{Lemma}
\newtheorem{cor}[thm]{Corollary}
\newtheorem{prop}[thm]{Proposition}

\theoremstyle{definition}

\newtheorem{df}[thm]{Definition}

\theoremstyle{remark}

\newtheorem{preremark}[thm]{Remark}
\newtheorem{preex}[thm]{Example}

\newenvironment{remark}{\begin{preremark}}{\qed\end{preremark}}
\newenvironment{ex}{\begin{preex}}{\qed\end{preex}}

\numberwithin{equation}{section}

\title[On a mixed Monge-Amp\`ere operator]{On a mixed Monge-Amp\`ere operator for quasiplurisubharmonic functions with analytic singularities}

\begin{document}
\date{\today}

\author{Richard L\"ark\"ang \& Martin Sera \& Elizabeth  Wulcan}
\address{R.\ L\"ark\"ang, E.\ Wulcan, Department of Mathematical Sciences\\Chalmers University of Technology and the University of Gothenburg\\412 96 Gothenburg\\SWEDEN}
\email{larkang@chalmers.se, wulcan@chalmers.se}
\address{M.\  Sera, Nagamori Institute of Actuators\\Kyoto University of Advanced Science\\Kyoto 615-8577\\JAPAN}
\email{sera.martin@kuas.ac.jp}

\thanks{The first and the third author were partially supported by the
  Swedish Research Council. The second author was supported by the
  Knut and Alice Wallenberg Foundation.}

\begin{abstract}
    We consider mixed Monge-Amp\`ere products of quasiplurisubharmonic functions with analytic singularities,
    and show that such products may be regularized as explicit one parameter limits of
    mixed Monge-Amp\`ere products of smooth functions, generalizing results of Andersson, B\l{}ocki
    and the last author in the case of non-mixed Monge-Amp\`ere products.
    Connections to the theory of residue currents, going back to Coleff-Herrera, 
    Passare and others, play an important role in the proof.
As a consequence we get an approximation of Chern and Segre currents of
certain singular hermitian metrics on vector bundles by smooth forms
in the corresponding Chern and Segre classes. 
\end{abstract}

\maketitle

\section{Introduction}\label{intro}
Classical pluripotential theory, going back to Bedford-Taylor, \cites{BT1,BT2}, gives a way of defining mixed Monge-Amp\`ere products like
\begin{equation}\label{huvudroll}
    dd^c u_r \wedge \dots \wedge dd^c u_1, 
\end{equation}
where $u_1,\ldots, u_r$ are locally bounded  plurisubharmonic (psh)
functions on a complex manifold $X$.  
Here and throughout $d^c=(\partial - \dbar)/(4\pi i)$. 
Let $u$ be a locally bounded psh function 
and let
$T$ be a closed positive current on $X$. Then 
\begin{equation}\label{pasklilja} 
dd^c u \wedge T := dd^c (uT)
\end{equation}
is a well-defined closed positive current. In particular one can give
meaning to mixed Monge-Amp\`ere products like \eqref{huvudroll} 
by inductively applying \eqref{pasklilja}. 
Theorem ~2.1 in \cite{BT2} asserts that \eqref{huvudroll} satisfies
the following monotone continuity: 
If $u^j_k$ are decreasing
sequences of psh functions
converging pointwise to $u_k$, then
\begin{equation}\label{monoton}
dd^cu^j_r\wedge\cdots\wedge dd^cu^j_1\to
dd^cu_r\wedge\cdots\wedge dd^cu_1.
\end{equation}

Demailly later extended this construction to the situation where 
the unbounded loci of the $u_i$ are small in a certain
sense, \cite{Dem}. 
For general psh functions there is no such canonical (mixed) Monge-Amp\`ere
product as \eqref{huvudroll}; e.g., one cannot expect 
\eqref{monoton} to hold in
general. 

\smallskip 

Recall that a psh function $u$ has \emph{analytic
  singularities}\footnote{See Remark ~\ref{rem:nonSmooth}.} if locally 
\begin{equation}\label{as}
  u=c\log|f|^2+v,
\end{equation}
where $c$ is a positive constant,  $f=(f_1,\dotsc,f_m)$ is a tuple of holomorphic functions,
and $v$ is smooth. 
In \cite{LRSW}, together with Raufi,  
we gave meaning to \eqref{huvudroll} for psh functions
$u_i$ with analytic singularities on $X$ by inductively
defining it as 
\begin{equation}\label{skata} 
dd^c u_{k} \wedge \cdots\wedge dd^c u_1:=
dd^c\big
(u_{k}\1_{X\setminus Z_{k}}dd^c u_{k-1}\wedge\cdots\wedge
dd^c u_1\big ),  
\end{equation}
where $Z_k$ is the unbounded locus of $u_k$, for $k=1,\ldots,
r$. 
Assuming that we have inductively defined $T:=dd^c u_{k-1}\wedge \cdots
\wedge dd^c u_1$, then for $u=u_{k}$ with unbounded locus $Z$ we define 
\begin{equation}\label{nyar} 
u\1_{X\setminus Z} T = 
\lim_{j\to \infty} u^j \1_{X\setminus Z} T, 
\end{equation} 
where $u^j$ is a sequence of smooth psh functions decreasing to
$u$. 
Propositions~3.2 and~3.4 in \cite{LRSW} assert that \eqref{nyar} has locally finite
mass and is independent of the
regularizing sequence $u^j$, and that 
\[
dd^c u \wedge \1_{X\setminus Z} T = dd^c (u\1_{X\setminus Z} T)
\]
is closed and positive and coincides with the classical
Bedford-Taylor-Demailly Monge-Amp\`ere product when this is defined. 
The definition of the product \eqref{skata} is a straightforward generalization of
previous work \cite{AW} by Andersson and the last author, where the
generalized Monge-Amp\`ere product $(dd^c u)^m$ was defined for psh functions $u$
with analytic singularities.

In \cite{LRSW} the generalized mixed Monge-Amp\`ere products
\eqref{skata} were used to construct Chern and Segre forms for certain singular
metrics on vector bundles, 
and in \cites{ASWY, AESWY} currents like these were
used to understand nonproper intersection theory in terms 
of currents.

\medskip 

The main goal of this paper is to prove a one parameter regularization
of the mixed
Monge-Amp\`ere products \eqref{skata}, similar to \eqref{monoton}. 
In fact, we will work in a slightly more general setting: 
Recall that a function $\varphi:X\to \R\cup\{-\infty\}$ is \emph{quasiplurisubharmonic (qpsh)}
if it is locally given as $\ph=u+a$, 
where $u$ is psh and $a$ is smooth.
We say that $\ph$ has \emph{analytic singularities} if $u$
has. 
Moreover, we
say that a closed current $T$  that is locally given as a sum
of currents \eqref{skata} multiplied by smooth closed $(p,p)$-forms has
\emph{analytic singularities}, see Definition
~\ref{current-with-as}. 
In \cite{LRSW}*{Lemma~3.5}, we showed that 
$\varphi\1_{X\setminus Z} T:= u\1_{X\setminus Z} T +
    a\1_{X\setminus Z} T,$
where $Z$ is the unbounded locus of $\varphi$, 
is independent of the
decomposition $\varphi=u+a$. 
It follows that 
$dd^c \varphi\wedge T= dd^c (\varphi\1_{X\setminus Z} T)$ is a well-defined current
with analytic singularities, and in particular we can inductively
define 
products 
\begin{equation}\label{hoppet}
dd^c \varphi_r\wedge \cdots\wedge dd^c \varphi_1, 
\end{equation}
if $\varphi_1, \ldots , \varphi_r$ are qpsh functions with analytic
singularities.

Since \eqref{monoton} does not hold in general 
one cannot expect convergence of any
decreasing regularizing sequences $\varphi_k^j$. 
For example, one can find smooth decreasing sequences of psh  
functions $u^j$ and $v^j$ converging to the same psh function
$u$ with analytic singularities, but where $(dd^c u^j)^2$ and $(dd^c v^j)^2$ converge to
different positive closed currents, see, e.g., Example~3.2 in
\cite{ABW}. 

\begin{df} \label{df:rho}
Let $\rho\colon \R \to \R$ be a smooth, convex, increasing function such that $\rho(t)$ is constant for $t \ll 0$ and such that $\rho(t) = t$ for $t \gg 0$.
Let $\rho_j(t) = \rho(t+j)-j$. 
\end{df}

Note, that if $\ph$ is a qpsh function with analytic singularities,
then $\rho_j\circ \varphi$ is a sequence of smooth functions
decreasing to $\ph$. 
In \cite{ABW}*{Theorem~1.1} it was proved that if $\varphi$ is a psh
function with analytic singularities, then 
\begin{equation}\label{himmelriket}
\lim_{j\to \infty} \big (dd^c (\rho_j \circ \varphi)\big )^m = (dd^c
\varphi)^m, 
\end{equation} 
and in \cite{B}*{Theorem~1} this was extended to the case when $\varphi$ is
qpsh. 
In \cite{A} the product $(dd^c u)^m$ was defined in the case when
$\varphi$ is of the form $\log|f|^2$ and a version of \eqref{himmelriket}
was proved in this case, see \cite{A}*{Proposition~4.4}.

It is not hard to see that
\eqref{hoppet} is not commutative in general, see, e.g., 
\cite{LRSW}*{Example~3.1} and 
therefore it cannot hold in general that 
\[ dd^c(\rho_{j_2}\circ\ph_2)\wedge dd^c(\rho_{j_1}\circ\ph_1)\to dd^c
\ph_2\wedge dd^c \ph_1
\]
as $j_1\to \infty$ and $j_2\to \infty$ independently, 
cf.\ Remark ~\ref{rmk:Lucas}.
The following definition is inspired by the residue theory due to 
Coleff and Herrera, \cite{CH}.

\begin{df}\label{df:adm-j}
    We say that a sequence $(j_1,\dots,j_r)\colon \N \to \R^r$ \textit{tends to $\infty$ along an admissible path},
    if for any $q \geq 0$, and $k=1,\dots,r-1$,
    \begin{equation*}
        j_k(\nu)-q\cdot j_{k+1}(\nu) \to \infty
    \end{equation*}
and $j_r(\nu) \to \infty$ as $\nu\to \infty$. 
\end{df}

\begin{ex}
The sequence $(j_1, j_2, \ldots, j_r)=(\nu^r, \nu^{r-1}, \ldots, \nu)$ 
tends to $\infty$ along an admissible path. 
\end{ex}

Our main result is the following generalization of
\eqref{himmelriket}. 

\begin{thm}\label{huvud} 
    Assume that $\varphi_1,\ldots, \varphi_r$ are qpsh functions with
    analytic singularities and let $\rho_j$ be as in Definition~\ref{df:rho}.
    If the sequence $(j_1,\dotsc,j_r)\colon \N\rightarrow \R^r$ tends to $\infty$ along an admissible path,
    then
    \begin{equation*}
        \lim_{\nu\rightarrow \infty}
        \big(dd^c(\rho_{j_r(\nu)}\circ\ph_r)\big)^{m_r} \wedge\dots\wedge   \big(dd^c(\rho_{j_1(\nu)}\circ\ph_1)\big)^{m_1} 
        =(dd^c \varphi_r)^{m_r}\wedge\cdots\wedge (dd^c \varphi_1)^{m_1}
    \end{equation*}
    for $m_1,\ldots, m_r\geq 1$. 
\end{thm}

Indeed, in the case when $r=1$ we just get back \eqref{himmelriket}. 
In fact, in \cites{ABW, B} the results are slightly more general; 
a more general
definition of analytic singularities is used, see Remark
~\ref{rem:nonSmooth}, and slightly more general
sequences $\rho_j$ are allowed, see Remark ~\ref{bobby}. 

\medskip

Inspired by \cite{ABW}*{Theorem~1.2}, in \cite{LRSW} we introduced
a formalism for global generalized mixed Monge-Amp\`ere operators. 
If $\varphi$ is a qpsh function with analytic
singularities and unbounded locus $Z$, 
$\theta$ and $\eta$ are closed $(1,1)$-forms, and $T$ is a current
with analytic singularities on $X$,  
we let 
\begin{equation}\label{sommarskuggan}
[\theta + dd^c\varphi]_\eta\wedge T :=
\theta \wedge \1_{X\setminus Z} T + dd^c\varphi \wedge \1_{X\setminus Z} T +
\eta\wedge \1_Z T.
\end{equation}
In fact, in \cite{LRSW} we only allowed $\varphi$ to be psh, but it is
not hard to see that the definition extends to qpsh functions; Lemma
~\ref{fjutt} asserts that   
 $[\theta + dd^c\varphi]_\eta\wedge T$ is a well-defined current
with analytic singularities that  is independent of the decomposition
of the current $\theta +dd^c\varphi$ as the sum of $\theta$ and
$dd^c\varphi$. 
In particular, if $\varphi_1,\ldots, \varphi_r$ are qpsh functions with
analytic singularities and $\theta_1,\ldots, \theta_r$ and $\eta_1,\ldots, \eta_r$ are
closed $(1,1)$-forms, 
we can give meaning to 
the global mixed Monge-Amp\`ere product  
\begin{equation}
\label{generalizedMA}
[\theta_r + dd^c\varphi_r]^{m_r}_{\eta_r}\wedge \cdots \wedge [\theta_1 +
dd^c\varphi_1]^{m_1}_{\eta_1}
\end{equation}
by letting 
$[\theta_1 +dd^c\varphi_1]_{\eta_1}= [\theta_1 +dd^c\varphi_1]_{\eta_1}
\wedge 1$ and inductively applying \eqref{sommarskuggan}. 
We have the following mass formula:
\begin{prop}\label{host}
    Assume that $X$ is compact. Moreover, assume that $\varphi_1,\ldots,
    \varphi_r$ are qpsh functions with analytic singularities and that 
$\theta_1,\ldots, \theta_r$ and $\eta_1,\ldots, \eta_r$ are 
    closed $(1,1)$-forms on $X$ such that $\theta_k-\eta_k=d\alpha_k$ for some smooth forms $\alpha_k$.
    Then,
    \begin{equation*}
        \int_X [\theta_r + dd^c\varphi_r]^{m_r}_{\eta_r}\wedge \cdots 
        \wedge [\theta_1 +dd^c\varphi_1]^{m_1}_{\eta_1} = 
        \int_X \theta^{m_r}_r\wedge\cdots\wedge\theta^{m_1}_1, 
    \end{equation*}%
    where $m_1+\cdots +m_r=\dim X$. 
\end{prop}
In the case when $r=1$ (and $\varphi_1$ is psh and $\eta_1=\theta_1$),
this is just Theorem~1.2 in \cite{ABW}, see
\cite{LRSW}*{Remark~3.6} and Remark~\ref{sommaren} below.

We have the following regularization result for the
products \eqref{generalizedMA} in the case when 
$\eta_k=\theta_k$. 
\begin{thm}\label{paron}
    Assume that $\varphi_1,\ldots, \varphi_r$ are qpsh functions with
    analytic singularities, that $\theta_1,\ldots, \theta_r$ are 
    closed $(1,1)$-forms, and that $\rho_j$ is as in Definition~\ref{df:rho}.
    If the sequence $(j_1,\dotsc,j_r)\colon \N\rightarrow \R^r$ tends to $\infty$ along an admissible path,
    then
   \begin{multline*}
                \lim_{\nu\rightarrow \infty}
        \big(\theta_r+dd^c(\rho_{j_r(\nu)}\circ\ph_r)\big)^{m_r}
        \wedge\dots\wedge
        \big(\theta_1+dd^c(\rho_{j_1(\nu)}\circ\ph_1)\big)^{m_1} \\
        =[\theta_r + dd^c\varphi_r]^{m_r}_{\theta_r}\wedge \cdots \wedge [\theta_1 +dd^c\varphi_1]^{m_1}_{\theta_1}.
        \end{multline*}
\end{thm}

In Section ~\ref{sec:huvud} we present a regularization result, Theorem
~\ref{paronmos}, for \eqref{generalizedMA} in the general case. 
Theorems ~\ref{huvud} and ~\ref{paron} are immediate consequences of
Theorem ~\ref{paronmos} below. In fact,  Theorem ~\ref{huvud} also follows immediately
from Theorem ~\ref{paron} by setting each $\theta_k=0$.

When $r=1$ Theorem ~\ref{paron} reads: if $\varphi$ is a qpsh function with
analytic singularities and $\theta$ is a closed $(1,1)$-form, then 
\begin{equation*}
\lim_{j\to\infty}
 \big (\theta+ dd^c (\rho_j\circ \varphi) \big )^m=
[\theta + dd^c\varphi]_{\theta}^m.
\end{equation*}
This is Theorem~1 in \cite{B}, except that the setting there is
slightly more general, cf.\ the discussion after Theorem
~\ref{huvud}. Also in \cite{B} the right hand side is denoted simply by
$(\theta + dd^c\varphi)^m$, see Remark ~\ref{sommaren}. 

\smallskip 

Mixed Monge-Amp\`ere products of qpsh functions with analytic
singularities are closely related to so-called residue currents in the
sense of Coleff-Herrera, \cite{CH}, and the proofs of our results are based on regularization results for residue
currents. In particular, we use a slightly modified result by the
first author and Samuelsson Kalm \cite{LS}.

\begin{remark} \label{rem:nonSmooth}
    In the literature, sometimes a more general definition of psh functions with analytic singularities is used than here,
    namely that in \eqref{as}, the function $v$ is just required to be locally bounded. 
    In the papers \cites{AW, ABW, LRSW, B} 
    this more general definition of psh and qpsh functions with analytic
    singularities is considered.  
    Also Proposition~\ref{host}  and the results in
    Section~\ref{MAsektion} below 
    work for this more general definition, while the smoothness of $v$ appears to be 
    essential in the proofs of Theorems~\ref{huvud} and \ref{paron}.
 \end{remark}

\medskip

The paper is organized as follows. In Section ~\ref{MAsektion} we
discuss the construction of the generalized mixed Monge-Amp\`ere operator 
from \cite{LRSW}. In particular, we give a proof of Proposition
~\ref{host}. We also relate our
products to
mixed non-pluripolar
Monge-Amp\`ere products in the sense of \cites{BT3, BEGZ} and rephrase
Proposition ~\ref{host} in terms of these.  
In Section ~\ref{sec:motivation} we give some background on
(regularization of) residue currents and show how they are related to
mixed Monge-Amp\`ere products of (q)psh functions with analytic
singularities. We also give a proof of a special case of Theorem
~\ref{huvud}. 
In Section ~\ref{sec:huvud} we prove Theorems ~\ref{huvud} and
~\ref{paron} and more generally Theorem ~\ref{paronmos}, and we also discuss some possible generalizations. 
Finally, in Section ~\ref{sec:Segre}, we present an application of
Theorem ~\ref{paron} to Chern and Segre currents for singular
hermitian metrics with analytic singularities as defined in
\cite{LRSW}. Corollary ~\ref{cor:application} asserts that these Chern and Segre currents are given as one parameter limits of smooth forms in the
corresponding Chern and Segre classes.

\medskip 
\noindent 
\textbf{Acknowledgement:} We would like to thank Mats Andersson and
Zbigniew B\l{}ocki for valuable discussions related to this paper.

\section{Mixed Monge-Amp\`ere products of qpsh functions with analytic
singularities}
\label{MAsektion}

In this section we give some further background on generalized mixed
Monge-Amp\`ere products of qpsh functions with analytic
singularities. As pointed out in the introduction, within this section
we allow psh and qpsh functions that have analytic singularities in
the less restrictive way, i.e., where we only require $v$ in the presentation
\eqref{as} to be bounded, cf.\ Remark ~\ref{rem:nonSmooth}. Throughout
the paper we will assume that $X$ is a complex manifold. 
Recall that
the \emph{unbounded locus} 
of a psh function $u$ on $X$ 
is the set of points $x\in X$ such that $u$ is unbounded in every
neighborhood of $x$. The unbounded locus of a qpsh function $\varphi$,
locally given as $\varphi=u+a$, is defined as the unbounded locus of
$u$. 
Note that if $u$ or $\varphi$ has analytic singularities, then the unbounded locus is an analytic set, locally defined by
$\{f=0\}$ where $u$ is given by \eqref{as}.

The construction of mixed Monge-Amp\`ere operators in \cite{LRSW} is slightly more
general than mentioned in the introduction. 
Assume 
that $u_1,\ldots, u_r$ are psh functions with analytic
singularities on $X$, with unbounded loci $Z_1,\ldots,
Z_r$, respectively. 
Moreover assume that $U_1,\ldots,
U_r\subset X$ are constructible sets contained in $X\setminus Z_1,
\ldots, X\setminus Z_r$, respectively. 
In \cite{LRSW}*{Section~3} we gave meaning to the product 
\begin{equation}\label{blomma2}
dd^c u_r\1_{U_r}\wedge\cdots\wedge dd^c u_1 \1_{U_1}, 
\end{equation}
by defining it recursively as 
\begin{equation}\label{insta2}
dd^c u_{k}\1_{U_k}\wedge \cdots\wedge dd^c u_1\1_{U_1}:=
dd^c\big
(u_{k}\1_{U_{k}}dd^c u_{k-1}\1_{U_{k-1}}\wedge\cdots\wedge
dd^c u_1\1_{U_1}\big )
\end{equation}
for $k=1,\ldots, r$. 
Here 
\begin{equation}\label{paskafton}
u_{k}\1_{U_{k}}dd^c u_{k-1}\1_{U_{k-1}}\wedge\cdots\wedge
dd^c u_1\1_{U_1}=
\lim_{j\to \infty} u_k^j \1_{U_{k}}dd^c u_{k-1}\1_{U_{k-1}}\wedge\cdots\wedge
dd^c u_1\1_{U_1}, 
\end{equation} 
where $u^j_k$ is a sequence of smooth psh functions decreasing to
$u_k$. 
Proposition~3.2 in \cite{LRSW} asserts that \eqref{paskafton} has locally finite
mass and is independent of the
regularizing sequence $u^j_k$, and that \eqref{insta2} is a closed positive current.

\begin{df}\label{current-with-as}
We say that a closed $(p,p)$-current has \emph{analytic singularities}
if it is locally of the form
\begin{equation*}
T=\sum \beta_i \wedge \1_{U_i}T_i,
\end{equation*}
where the sum is finite, $\beta_i$ are closed 
forms, $U_i\subset X$ are constructible sets, and $T_i$
are currents of the form \eqref{blomma2} or $T_i=1$.  
\end{df}

We should remark that this definition extends (in a non-essential way)
the 
definition in \cite{LRSW}*{Section~3}. There a current with analytic
singularities refers to ($\1_U$ times) a current of the form
\eqref{blomma2}.

Note, in particular, that if $T$ is a 
current with analytic singularities, $u$ is a psh function with
analytic singularities with unbounded locus $Z$, and $U$ is a
constructible set contained in $X\setminus Z$, then
$dd^c u\wedge \1_U T:=dd^c(u\1_UT)$ is a well-defined current with
analytic singularities, cf.\ Remark~3.3 in \cite{LRSW}.

In \cite{LRSW}*{Lemma~3.5}, it was proved that if $T$ is a current
with analytic singularities, $\varphi=u+a$ is a qpsh function with
analytic singularities with unbounded locus $Z$, and $U\subset
X\setminus Z$ is a constructible set, then 
\begin{equation}\label{hemmavid}
\varphi\1_{U} T:= u\1_{U} T +
    a\1_{U} T,
\end{equation}  
is independent of the decomposition $\varphi=u+a$.
It follows that 
$dd^c \varphi\wedge \1_U T:= dd^c (\varphi\1_{U} T)$ is a well-defined current
with analytic singularities. 
In particular, we can inductively
define generalized mixed Monge-Amp\`ere products 
\begin{equation}\label{katta}
dd^c \varphi_r\1_{U_r}\wedge \cdots\wedge dd^c \varphi_1\1_{U_1}, 
\end{equation}
if $\varphi_i$ are qpsh functions with analytic
singularities with unbounded loci $Z_i$ and 
$U_i\subset X\setminus Z_i$ are constructible sets. 

\begin{remark}\label{hund}
Assume that $\pi:X'\to X$ is a holomorphic modification and that
$\varphi_1,\ldots, \varphi_r$ are qpsh functions with analytic
singularities on $X$. 
Then $\pi^*\varphi_1, \ldots, \pi^*\varphi_r$
are qpsh functions with analytic singularities on $X'$. Moreover,
using that 
$\alpha\wedge \pi_* \mu = \pi_* (\pi^* \alpha\wedge \mu)$
for any smooth form $\alpha$ on $X$ and current $\mu$ on $X'$, and that 
$\1_U\pi_* \mu = \pi_* (\1_{\pi^{-1}U} \mu)$
for any constructible set $U\subset X$ and any positive closed (or
normal) current $\mu$ on $X'$, it follows from the construction
that, if $U_1,\ldots,U_r \subset X$ are constructible sets contained in the complement of the unbounded
loci of $\varphi_1,\dots,\varphi_r$, respectively, then 
\begin{equation*}
 dd^c \varphi_r\1_{U_r}\wedge \cdots\wedge dd^c \varphi_1\1_{U_1} 
=
\pi_* \big (dd^c \pi^*\varphi_r\1_{\pi^{-1}U_r}\wedge \cdots\wedge
dd^c \pi^*\varphi_1\1_{\pi^{-1}U_1} \big ) 
\end{equation*} 
More generally it follows that for any current $T$ with analytic
singularities on $X$ 
there is a current $T'$ with analytic singularities on $X'$ such that
$T=\pi_* T'$. 
\end{remark} 

\begin{remark}\label{katt}
Note that $dd^c \varphi\wedge\1_U T$ only depends on the current $dd^c\varphi$ and not
on the particular choice of potential $\varphi$. Indeed, assume that
$\varphi_1=\varphi_2+h$, where $dd^c h=0$. Then $h$ is smooth and thus 
\[
dd^c \varphi_1\wedge \1_U T = dd^c (\varphi_2+h) \wedge\1_U T = dd^c \varphi_2\wedge\1_U
T + dd^c h \wedge \1_U T = dd^c \varphi_2\wedge \1_U T, 
\]
where the second equality follows since \eqref{hemmavid} is
independent of the decomposition $\varphi=u+a$.
\end{remark}

As in the introduction we will use the shorthand notation
\begin{equation}\label{applen} 
dd^c \varphi_r\wedge\cdots\wedge dd^c \varphi_1 
= 
dd^c \varphi_r \1_{X\setminus Z_r} \wedge 
\cdots\wedge dd^c \varphi_1 \1_{X\setminus Z_1}, 
\end{equation}
where $Z_k$ is the unbounded locus of $\varphi_k$.
This product is neither commutative nor additive in any of the factors (except for the right-most one), cf.~\cite{LRSW}*{Example~3.1}.

Let $\varphi$ be a qpsh function with analytic singularities with
unbounded locus $Z$, and let $\rho_j$ be as in Definition
~\ref{df:rho}. 
Since $\rho_j\circ \varphi$ is constant in a neighborhood of $Z$, 
\begin{equation}\label{karius}
\lim_{j\to\infty} dd^c (\rho_j\circ \ph) \wedge T 
= 
\lim_{j\to\infty} dd^c (\rho_j\circ \ph) \wedge \1_{X\setminus Z} T 
=
dd^c \varphi \wedge \1_{X\setminus Z} T. 
\end{equation}
In particular, with the shorthand notation \eqref{applen}, we get 
\begin{equation}\label{bacill}
    \lim_{{j_r}\to \infty}\cdots \lim_{j_1\to \infty}   dd^c(\rho_{j_r}\circ\ph_r)\wedge\dots\wedge
        dd^c(\rho_{j_1}\circ\ph_1) = 
dd^c\ph_r\wedge\dots\wedge
        dd^c \ph_1. 
\end{equation} 
In fact, from this it follows that \eqref{applen} coincides with the
classical Bedford-Taylor-Demailly product when this is defined,
cf.\ (the proof of) Proposition~3.4 in \cite{LRSW}. 
In particular, \eqref{applen} coincides with the classical product
$dd^c\varphi_r\wedge \cdots \wedge dd^c\varphi_1$ 
outside $Z_1\cup\cdots\cup Z_r$. 

\medskip

The following lemma shows that
\eqref{sommarskuggan} is independent of the decomposition of the
current $\theta+dd^c\varphi$ as the sum of $\theta$ and $dd^c\varphi$. 

\begin{lma}\label{fjutt}
Let $\varphi_1,\varphi_2$ be qpsh functions with analytic
singularities, let $\theta_1, \theta_2, \eta$ be closed
$(1,1)$-forms, and let $T$ be a current with analytic singularities.
Assume that 
$\theta_1 + dd^c\varphi_1 = \theta_2 + dd^c\varphi_2.$
Then 
\begin{equation}\label{struntsak}
[\theta_1 + dd^c\varphi_1]_\eta\wedge T =[\theta_2 +
dd^c\varphi_2]_\eta\wedge T. 
\end{equation}
\end{lma}

\begin{proof}
It is enough to prove \eqref{struntsak} locally in
$X$ and thus we may assume that the $dd^c$-lemma holds on $X$.
Note that $\theta_1-\theta_2=dd^c(\varphi_2-\varphi_1)$ is smooth and
$d$-closed.
Therefore, by the $dd^c$-lemma, there is a smooth
function $a$ such that $\theta_1-\theta_2=dd^c a$, i.e.~
$dd^c(\varphi_1+a)=dd^c \varphi_2.$
In particular, the difference of $\ph_1+a$ and $\ph_2$ is pluriharmonic and thus smooth, so
the unbounded loci of $\varphi_1$ and $\varphi_2$ coincide;
let us denote this set by $Z$. 
Now
\begin{multline*} 
[\theta_1 + dd^c\varphi_1]_\eta\wedge T - [\theta_2 +
dd^c\varphi_2]_\eta\wedge T 
= \\
\theta_1 \wedge \1_{X\setminus Z} T + dd^c\varphi_1 \wedge \1_{X\setminus Z} T 
- \theta_2 \wedge \1_{X\setminus Z} T - dd^c\varphi_2 \wedge
\1_{X\setminus Z} T 
=\\
 dd^c a \wedge \1_{X\setminus Z} T 
+ dd^c\varphi_1 \wedge \1_{X\setminus Z} T 
- dd^c\varphi_2 \wedge\1_{X\setminus Z} T 
=\\
dd^c (a + \varphi_1) \wedge\1_{X\setminus Z} T - dd^c\varphi_2 \wedge\1_{X\setminus Z} T = 0, 
\end{multline*} 
where the third equality follows since \eqref{hemmavid} is independent
of the decomposition $\varphi=u+a$,
and the last equality follows in view of Remark~\ref{katt} since
$dd^c(\varphi_1+a)=dd^c \varphi_2$. 
\end{proof}

We obtain the following result regarding the $d$- and
$dd^c$-cohomology for generalized Monge-Amp\`ere products; a version
of this appeared as Proposition~4.3 in \cite{LRSW}.  
\begin{prop}\label{vinter}
Assume that $\varphi$ is a qpsh function with analytic singularities, 
that $\theta$ and $\eta$ are closed $(1,1)$-forms, 
and that $T$ is a current with analytic singularities.
Moreover, assume that $\theta-\eta=d\alpha$, where $\alpha$ is a smooth
form. Then, there is a current $S$ such that 
\begin{equation}\label{valparaiso}
[\theta + dd^c\varphi]_\eta\wedge T = \theta \wedge T + dS.
\end{equation}
If moreover $\theta-\eta=dd^c a$, where $a$ is a smooth
function, then there is a current $S'$ such that 
\begin{equation}\label{skriva} 
[\theta + dd^c\varphi]_\eta\wedge T = \theta \wedge T + dd^cS'.
\end{equation}
\end{prop}

\begin{proof}
Since $\theta-\eta=d\alpha$, 
\begin{multline*}
[\theta + dd^c\varphi]_\eta\wedge T = 
\theta \wedge \1_{X\setminus Z} T + dd^c\varphi \wedge \1_{X\setminus Z} T +
\eta\wedge \1_Z T
=\\
\theta \wedge  T + dd^c(\varphi \1_{X\setminus Z} T) +
(\eta-\theta) \wedge \1_Z T
= 
\theta \wedge  T + d \big (d^c(\varphi \1_{X\setminus Z} T) -
\alpha \wedge \1_Z T \big ), 
\end{multline*}
where we in the last equation used that $\1_Z T$ is closed by the Skoda-El Mir theorem.
Thus \eqref{valparaiso} holds with $S= d^c(\varphi \1_{X\setminus Z}
T) -
\alpha \wedge \1_Z T$. 

If $\theta-\eta=dd^c a$, then by the same arguments,
\eqref{skriva} holds with $S'=\varphi \1_{X\setminus Z} T -a \1_Z T$.
\end{proof}

Now Proposition ~\ref{host} follows immediately from Proposition
~\ref{vinter}.

\begin{remark}
Given psh functions $u_1,\ldots, u_r$, the \emph{mixed
  non-pluripolar Monge-Amp\`ere product} 
\begin{multline}\label{npp}
\big \langle (dd^c u_r)^{m_r}\wedge \cdots \wedge (dd^c u_1)^{m_1}
\big \rangle 
= \\
\lim_{j\to \infty} \1_{\bigcap_i \{u_i > -j\}}
\big  (dd^c \max (u_r, -j) \big )^{m_r} 
\wedge \cdots \wedge 
\big  (dd^c \max (u_1, -j) \big )^{m_1}, 
\end{multline} 
introduced in \cites{BT3, BEGZ},
is a closed positive current that does not charge any pluripolar set
and that is well-defined if the unbounded loci of
$u_i$ are small in a certain sense, see \cite{BEGZ}*{Section~1.2}, in
particular, if the $u_i$ have analytic singularities.

Given closed $(1,1)$-forms $\theta_i$ and $\theta_i$-psh functions $\varphi_i$, i.e.,
$\theta_i+dd^c\varphi_i\geq 0$ for $i=1,\ldots, r$ 
one can extend \eqref{npp} to define the  non-pluripolar product 
$\langle 
(\theta_r + dd^c\varphi_r)^{m_r}\wedge\cdots\wedge
(\theta_1+dd^c\varphi_1)^{m_1}\rangle$. 
If the $\varphi_i$ have analytic singularities with unbounded loci $Z_i$
and we let $Z=Z_1\cup\cdots\cup Z_r$, it follows from the construction
that 
\[
\big \langle 
(\theta_r + dd^c\varphi_r)^{m_r}\wedge\cdots\wedge
(\theta_1+dd^c\varphi_1)^{m_1}\big \rangle
=
\1_{X\setminus Z} 
[\theta_r + dd^c\varphi_r]^{m_r}_{\eta_r}\wedge \cdots 
        \wedge [\theta_1 +dd^c\varphi_1]^{m_1}_{\eta_1}, 
 \]
if $\eta_1,\ldots, \eta_r$ are closed $(1,1)$-forms, cohomologous to
$\theta_1,\ldots, \theta_r$, respectively. 
Thus the mass formula Proposition ~\ref{host} can be rephrased as 
\begin{multline*}
\int_X
\big \langle 
(\theta_r + dd^c\varphi_r)^{m_r}\wedge\cdots\wedge
(\theta_1+dd^c\varphi_1)^{m_1}\big \rangle
=\\
\int_X \theta_r^{m_r}\wedge\cdots\wedge \theta_1^{m_1} - 
\int_X\1_Z [\theta_r + dd^c\varphi_r]^{m_r}_{\eta_r}\wedge \cdots 
        \wedge [\theta_1 +dd^c\varphi_1]^{m_1}_{\eta_1},
\end{multline*} 
cf.\ \cite{ABW}*{Equation~(5.5)}. 
\end{remark}

\begin{remark}\label{sommaren} 
Note that 
$[\theta+dd^c\varphi]_\theta\wedge T= \theta \wedge T + dd^c
\varphi\wedge \1_{X\setminus Z}T$. In particular, 
\[
[\theta+dd^c\varphi]_\theta^m = 
\big (\theta + dd^c
\varphi \1_{X\setminus Z}\big )^m =
\sum_{\ell=0}^m \binom m \ell \theta^{m-\ell}\wedge (dd^c \varphi)^\ell,
\]
where we use the shorthand notation \eqref{applen} in the rightmost expression. In \cite{B} this global
Monge-Amp\`ere product was just denoted by $(\theta+dd^c\varphi)^m$. 
We prefer to use the notation
$[\theta+dd^c\varphi]_\theta^m$ to emphasize that it depends 
not only on the current $\theta+dd^c\varphi$ but also on the
decomposition as the sum of $\theta$ and $dd^c\varphi$, cf.\ Theorem~3
in \cite{B} and the following discussion. 

Alternatively, 
\begin{equation}\label{boot}
[\theta+dd^c\varphi]_\theta^m = 
\big ((\theta + dd^c
\varphi) \1_{X\setminus Z} + \theta\1_Z \big )^m =
(\theta + dd^c \varphi)^m + 
\sum_{\ell=0}^{m-1} \theta^{m-\ell}\wedge \1_Z (\theta+dd^c \varphi)^\ell. 
\end{equation} 
In particular, it follows that $[\theta+dd^c\varphi]_\theta^m$ equals the ordinary
Monge-Amp\`ere product $(\theta + dd^c \varphi)^m$, if $\varphi$ is
locally bounded. 
In \cite{ABW} the mass formula Theorem ~1.2 was formulated in terms of
the right-hand side of \eqref{boot}, see \cite{LRSW}*{Remark~3.6}. 
\end{remark}

\begin{remark}\label{kaktus} 
Assume that $L\to X$ is a holomorphic line bundle. We say that a
positive hermitian singular (in the sense of Demailly \cite{Dem2})
metric $e^{-\phi}$ on $L$ has
\emph{analytic singularities} if the local weights $\phi$ are psh functions
with analytic singularities.  Since two local weights differ by a
pluriharmonic function the first Chern form $dd^c \phi$ is a
well-defined closed positive current on $X$. 

Let $e^{-\psi}$ be a smooth metric on $L$ with first Chern form
$\theta=dd^c\psi$. Then $\varphi:= \phi-\psi$ is a well-defined qpsh
function on $X$ and 
$dd^c \varphi=dd^c\phi -\theta$, and thus if $T$ is a current with
analytic singularities on $X$, we can write  
\begin{equation}\label{baktus}
[dd^c\phi]_\theta\wedge T := [\theta + dd^c\varphi]_\theta\wedge T.
\end{equation}
In particular, if $\phi_1,\ldots, \phi_r$ are positive hermitian
metrics with analytic singularities on $L$ and $\theta_1,\ldots,
\theta_r$ are Chern forms of smooth metrics $e^{-\psi_1}, \ldots,
e^{-\psi_r}$ on $L$, we can write
\begin{equation}\label{lakrits}
[dd^c\phi_r]_{\theta_r}^{m_r}\wedge\cdots\wedge
[dd^c\phi_1]_{\theta_1}^{m_1} 
=
[\theta_r + dd^c\varphi_r]_{\theta_r}^{m_r}\wedge\cdots\wedge
[\theta_1+dd^c\varphi_1]_{\theta_1}^{m_1},
\end{equation} 
where $\varphi_i=\phi_i-\psi_i$, cf. \cite{LRSW}*{Section~4}.
\end{remark}

\section{Residue currents}
\label{sec:motivation}

In this section we give some background on (regularizations of) residue currents and relate
them to certain mixed Monge-Amp\`ere products. In particular we prove
a special case of Theorem ~\ref{huvud}.

Throughout this paper, by a \emph{cut-off function} we mean
a function $\chi : \R_{\geq 0} \to \R_{\geq 0}$ which is smooth, increasing and such 
that $\chi(t) \equiv 0$ for $t \ll 1$ and $\chi(t) \equiv 1$ for $t
\gg 1$.

In \cite{AW2} was introduced a class of so-called
\emph{pseudomeromorphic} currents that includes all smooth forms, is
closed under multiplication with smooth forms
and the following operations:
If $f$ is a holomorphic function, $Z=\{f=0\}$, $\chi$
is a cut-off function, $\chi_\epsilon:= \chi(|f|^2/\epsilon)$, 
and $T$ is a pseudomeromorphic current on $X$,
then the following are well-defined
pseudomeromorphic currents:
\begin{multline} \label{PMproducts}
    \begin{gathered}
        \frac{1}{f} ~T := \lim_{\epsilon \to 0} \frac{\chi_\epsilon}{f}~T,\quad 
    \dbar\frac{1}{f}\wedge T := \lim_{\epsilon \to 0} \frac{\dbar\chi_\epsilon}{f} \wedge T \text{ and } \quad 
    \1_{X \setminus Z} T := \lim_{\epsilon \to 0} \chi_\epsilon T, 
    \end{gathered}
\end{multline}
see also \cite{AW3}. 
Since $\1_{X \setminus Z} T=T$
outside of $Z$, and $\dbar\chi_\epsilon$ has its support outside of
$Z$, it follows that 
\begin{equation} \label{PMrestrictions}
    \frac{1}{f} f T = \1_{X \setminus Z} T \quad \text{ and }\quad 
    \dbar\frac{1}{f} \wedge \1_{X \setminus Z} T = \dbar\frac{1}{f} \wedge T.
\end{equation}

In particular, if $f_1,\dots,f_r$ are holomorphic functions, then 
\begin{equation} \label{eq:CH}
    \dbar \frac{1}{f_r} \wedge \dots \wedge \dbar \frac{1}{f_1} :=
    \lim_{\epsilon_r \to 0} \cdots \lim_{\epsilon_1 \to 0} P_\epsilon,
\end{equation}
where
\begin{equation} \label{eq:Pdef}
  P_\epsilon = \frac{\dbar\chi_{r,\epsilon_r}}{f_r} \wedge \dots
  \wedge  \frac{\dbar\chi_{1,\epsilon_1}}{f_1}, 
\end{equation}
$\epsilon=(\epsilon_1,\dots,\epsilon_r)$, and 
\begin{equation}\label{singalong}
\chi_{k,\epsilon} =
\chi(|f_k|^2/\epsilon), 
\end{equation} 
is a well-defined pseudomeromorphic current. 
Products like these were first defined by Coleff and Herrera,
\cite{CH}, and therefore, \eqref{eq:CH} 
is often referred to as the \emph{Coleff-Herrera product} of $f_1,\dots,f_r$.
The products in \cite{CH} were defined in a slightly different way,
taking one parameter limits along certain so-called admissible paths 
instead of iterated limits like in \eqref{eq:CH}. 

\begin{df}\label{df:adm-eps}
    We say that a sequence $(\epsilon_1,\dots,\epsilon_r)\colon \N \to \R_{> 0}^r$ \textit{tends to $0$ along an admissible path},
    if for any $q \geq 0$, and $k=1,\dots,r-1$,
    \begin{equation*}
        \epsilon_k(\nu)/\epsilon_{k+1}^q(\nu) \to 0
    \end{equation*}
and $\epsilon_r(\nu) \to 0$ as $\nu\to \infty$.
\end{df}

Given a sequence $(j_1,\dots,j_r)\colon \N\rightarrow\R^r$, let
$(\epsilon_1,\dots,\epsilon_r)\colon \N \to \R_{> 0}^r$ be the
sequence defined by $\epsilon_k:= e^{-j_k}$ for
$k=1,\dots,r$. 
Then note that $(\epsilon_1,\dots,\epsilon_r)$ tends to $0$ along an
admissible path if and only if $(j_1,\dots,j_r)$ tends to $\infty$
along an admissible path, see Definition ~\ref{df:adm-j}. 
If $(\epsilon_1,\dots,\epsilon_r)$ tends to $0$ along an admissible path, then it follows by \cite{LS}*{Theorem~2} that 
\begin{equation}\label{eq:CH2}
\lim_{\nu\to\infty} P_{(\epsilon_1(\nu),\dots,\epsilon_r(\nu))} = 
\lim_{\epsilon_r' \to 0} \cdots \lim_{\epsilon_1' \to 0} P_{(\epsilon_1',\dots,\epsilon_r')},
\end{equation}
where $P_\epsilon$ is defined by \eqref{eq:Pdef}.
The left-hand side thus provides a regularization of
\eqref{eq:CH} as a one parameter limit of smooth forms.

To be precise, in \cite{CH}, the product $\dbar
(1/f_r)\wedge\cdots\wedge \dbar (1/f_1)$ is defined as the limit of $P_\epsilon$ along admissible paths, 
but where $\chi=\chi_{[1,\infty)}$ is the characteristic function of
$[1,\infty)$ 
and the factor $\dbar\chi_{r,\epsilon_r} \wedge \dots \wedge \dbar
\chi_{1,\epsilon_1}$ in $P_\epsilon$ then should be interpreted as
the current of integration along $\cap \{ |f_k|^2=\epsilon_k \}$.
By combining ideas from \cite{CH} and \cite{P} one can show that
\eqref{eq:CH} in fact coincides with Coleff-Herrera's original
definition, see \cite{LS}*{Section~1}; in particular, this follows from
Theorem~11 in \cite{LS}.

\smallskip

Let $\varphi_k = \log |f_k|^2$, where $f_k$ is a holomorphic function, 
and let $Z_k=\{f_k=0\}$. Then the mixed
Monge-Amp\`ere product \eqref{hoppet} is closely related to the
Coleff-Herrera product \eqref{eq:CH}. 
Formally, if $T$ is a pseudomeromorphic current, in view of
\eqref{PMproducts}, 
\begin{equation}\label{verbal}
dd^c \varphi_k \wedge \1_{X\setminus Z_k} T=\frac{1}{2\pi i} ~\dbar \partial \varphi_k \wedge \1_{X\setminus Z_k} T=
\frac{1}{2\pi i} ~\dbar \frac{1}{f_k}\wedge \partial f_k \wedge T
\end{equation}
and so, formally, 
\begin{equation} \label{eq:severalFactors2}
    dd^c \varphi_r \wedge \dots \wedge dd^c \varphi_1 =
    \dbar\frac{1}{f_r} \wedge \dots \wedge \dbar \frac{1}{f_1} \wedge
    \Theta, ~~~~ \text{ where } \Theta =  \frac{1}{(2\pi i)^r} 
\partial f_1 \wedge \dots \wedge \partial f_r. 
\end{equation}
To give a rigorous proof of \eqref{eq:severalFactors2}, let $\rho$ and
$\rho_j$ be as in Definition~\ref{df:rho} and let $\chi = \rho' \circ \log$. Then note that $\chi$ is a cut-off function and $(\rho_j'\circ \log)(t) = \chi(te^j)$. 
Then 
\begin{equation}\label{gronrutig}
\rho_j'\circ \varphi_k=\rho_j'(\log |f_k|^2) = 
\chi(|f_k|^2 e^j) = \chi_{k, e^{-j}}
\end{equation} 
see \eqref{singalong}.  
Thus 
\begin{equation}\label{blarutig} 
\partial(\rho_j \circ \varphi_k)
= 
\rho_j'\circ \varphi_k ~ \partial \varphi_k 
=
\chi_{k, e^{-j}}
\frac{\partial f_k}{f_k}. 
\end{equation}
Since $\partial f_k /f_k$ is holomorphic on the support of 
$\chi_{k, e^{-j}}$ 
it follows that 
\begin{equation}\label{gulrutig} 
    dd^c (\rho_j \circ \varphi_k) = \frac{1}{2\pi i}
\frac{\dbar\chi_{k, e^{-j}}}{f_k} \wedge \partial f_k. 
\end{equation}
Now, let $(\epsilon_1,\ldots, \epsilon_r)$ be defined by
$\epsilon_k=e^{-j_k}$. Then 
\begin{equation}\label{halsont}
        dd^c(\rho_{j_r}\circ\ph_r)\wedge\dots\wedge
        dd^c(\rho_{j_1}\circ\ph_1) = P_\epsilon \wedge \Theta, 
\end{equation} 
cf.\ \eqref{eq:CH} and \eqref{eq:severalFactors2}. 
Taking iterated limits $\lim_{j_r\to \infty}\cdots \lim_{j_1\to\infty}$ of 
both sides of \eqref{halsont}, in view of \eqref{bacill} and
\eqref{eq:CH}, we get \eqref{eq:severalFactors2}. 

\begin{remark}\label{delsjon}
If $\chi$ is a
cut-off function, then note that $\rho(t):=\int_0^t \chi(e^s)
ds +c$ is as in Definition~\ref{df:rho} for an appropriate choice of
constant $c$ and that $\rho'(\log t)=\chi(t)$. 
\end{remark}

Note that Theorem~\ref{huvud} in this case, when
$\varphi_k=\log|f_k|^2$ and $m_k=1$ for $k=1,\ldots, r$, follows
directly from \eqref{halsont}, \eqref{eq:CH2}, and
\eqref{bacill}, using that
$(j_1,\ldots,j_r)$ tends to $\infty$ along an admissible path if and
only if $(\epsilon_1,\ldots, \epsilon_r)$ tends to $0$ along an
admissible path. 

\begin{remark}\label{bobby} 
The reason that we require $\rho$ and $\rho_j$ in Definition ~\ref{df:rho} to be
slightly more restrictive than in \cites{ABW, B} is that then $\chi$
defined above is a cut-off function, which is used in, e.g., \cites{P,LS}. Possibly
the results (we need) in \cite{LS} could be extended to more
general $\chi$ that would correspond to more general $\rho$. 
\end{remark}

Next, let us consider functions of the form $\varphi_k = c_k \log
|f_k|^2 + v_k$, where $c_k>0$, $f_k$ is a single holomorphic function, and $v_k$ is smooth.
In fact, after a principalization and resolution of singularities, any qpsh function with 
analytic singularities is of this form. 
Then formally, using \eqref{PMrestrictions},  
\begin{equation*}
dd^c \varphi_k \wedge \1_{X\setminus Z_k} T
=
\left ( \frac{c_k}{2\pi i} ~\dbar \frac{1}{f_k}\wedge \partial f_k
+\frac{1}{f_k}\cdot f_k dd^c v_k \right ) \wedge T, 
\end{equation*}
cf.\ \eqref{verbal}, so that, formally,  
\begin{multline} \label{MMA-CHP1}
      dd^c\varphi_r \wedge\dots\wedge dd^c\varphi_1= 
\\
\left ( \frac{c_r}{2\pi i} ~\dbar \frac{1}{f_r}\wedge \partial f_r
+\frac{1}{f_r}\cdot f_r dd^c v_r \right )
\wedge\cdots\wedge  
\left ( \frac{c_1}{2\pi i} ~\dbar \frac{1}{f_1}\wedge \partial f_1
+\frac{1}{f_1}\cdot f_1 dd^c v_1 \right )
  \end{multline}
The right-hand side of \eqref{MMA-CHP1} may be approximated in a
similar way as above, cf.\ \eqref{ddcRho} below, and Theorem~\ref{huvud} in this situation may then be proved
using Proposition~\ref{LSvariant}, which is a generalization of
\eqref{eq:CH2} that allows for products of factors which are either
$\dbar(1/f_k)$ or $1/f_k$.

\section{Regularizations of mixed Monge-Amp\`ere products}
\label{sec:huvud}

In this section we prove Theorems ~\ref{huvud} and 
~\ref{paron}. 
In fact, we prove the following more general result. 
\begin{thm}\label{paronmos}
    Assume that $\varphi_1,\ldots, \varphi_r$ are qpsh functions with
    analytic singularities, that $\theta_1,\ldots, \theta_r$ and $\eta_1,\ldots, \eta_r$ are 
    closed $(1,1)$-forms, that $m_1,\dots,m_r \geq 1$, and that $\rho_j$ is as in Definition~\ref{df:rho}.
    Let
    \begin{equation}\label{herde}
    	\alpha_j^{(k)}=\big(\eta_k+\rho'_j\circ\ph_k\cdot(\theta_k-\eta_k)
        +dd^c(\rho_j\circ\ph_k)\big)^{m_k}.
            \end{equation}
    Assume that the sequence $(j_1,\dotsc,j_r)\colon \N\rightarrow
    \R^r$ tends to $\infty$ along an admissible path. 
    Then
    \begin{equation} \label{eq:Amazon}
        \lim_{\nu\rightarrow \infty} \alpha^{(r)}_{j_r(\nu)}\wedge \dots \wedge\alpha^{(1)}_{j_1(\nu)}
        =[\theta_r + dd^c\varphi_r]^{m_r}_{\eta_r}\wedge \cdots \wedge [\theta_1 +dd^c\varphi_1]^{m_1}_{\eta_1}.
    \end{equation}
\end{thm}

Theorem ~\ref{huvud} then corresponds to $\theta_i=\eta_i=0$ and
Theorem ~\ref{paron} corresponds to $\theta_i=\eta_i$ for
$i=1,\ldots, r$.  

The proof is essentially an elaboration of the proof of the special
case of Theorem ~\ref{huvud} in the previous section. Before giving
the proof we need some preparatory results. 
First, let us assume that 
$\varphi_k$ is of the form 
\begin{equation}\label{allihopa}
\varphi_k = c_k \log |f_k|^2 + v_k, 
\end{equation}
where $c_k$ is a positive constant, $f_k$ is a tuple of holomorphic
functions, and $v_k$ is smooth. 
Let $\rho$ and $\rho_j$ be as in Definition~\ref{df:rho}.
Let $\chi$ be the cut-off function $\chi := \rho \circ \log$,
let $\epsilon_j:=e^{-j}$, and let 
\begin{equation}\label{boman}
\chi_{k,\epsilon} :=
\chi(|f_k|^{2c_k}e^{v_k}/\epsilon), 
\end{equation}
cf.\ \eqref{singalong}\footnote{Note that \eqref{singalong}
  corresponds to $c_k=1$ and $v_k=0$ in \eqref{allihopa}.}. 
Then,
\begin{equation}\label{hundra}
\rho_j'\circ \varphi_k=\rho_j'(c_k\log |f_k|^2+v_k) = (\rho_j'\circ\log)(|f_k|^{2c_k} e^{v_k}) =
\chi(|f_k|^{2c_k}e^{v_k} e^j) =\chi_{k,\epsilon_j},   
\end{equation} 
cf.\ \eqref{gronrutig}, and thus 
 \begin{equation*}
    	\alpha_j^{(k)}=\big(\eta_k+\chi_{k,\epsilon_j}\cdot(\theta_k-\eta_k)
        +dd^c(\rho_j\circ\ph_k)\big)^{m_k}. 
\end{equation*}

Next, assume that $\ph_k$ is a qpsh function of the form \eqref{allihopa}, but
    where $f_k$ is a single holomorphic function.
Also, let us drop the index $k$ and assume that $\varphi$ is a function of
the form 
\begin{equation*}
\varphi = c\log |f|^2+v, 
\end{equation*} 
where $f$ is a holomorphic function, $c>0$, and $v$ is smooth, and
write $\chi_{\epsilon} =
\chi(|f|^{2c}e^v/\epsilon)$. Moreover, let $Z=\{f=0\}$ denote the
unbounded locus of $\varphi$. 
Then it follows from \eqref{hundra} that 
\begin{equation*}
\partial(\rho_j \circ \varphi)
= 
\rho_j'\circ \varphi ~ \partial \varphi 
=
\chi_{\epsilon_j}\cdot\left(c\frac{\partial f}{f} + \partial
  v\right), 
\end{equation*} 
cf.\ \eqref{blarutig}. 
Since $\partial f /f$ is holomorphic on the support of $\dbar
\chi_{\epsilon_j}$ it follows that 
\begin{equation} \label{ddcRho}
    dd^c (\rho_j \circ \varphi) = \frac{1}{2\pi i}\dbar\left( \chi_{\epsilon_j} \left(c\frac{\partial f}{f}+\partial v\right) \right) =
    \dbar \chi_{\epsilon_j} \wedge \frac{1}{2\pi i}\left(c\frac{\partial f}{f} + \partial v\right)
    + \chi_{\epsilon_j} dd^c v, 
\end{equation}
cf.\ \eqref{gulrutig}.

\begin{lma} \label{lma:residueVsMA}
   Assume that $T$ is a current with analytic singularities. 
    Then,
    \begin{equation}\label{eq:partial-ddcRho}
\lim_{j\to \infty}\dbar \chi_{\epsilon_j} \wedge \frac{1}{2\pi i}\left( c\frac{\partial f}{f}+\partial v\right) \wedge T\to dd^c (c\log |f|^2) \wedge \1_{X\setminus Z} T.
    \end{equation}
    \end{lma}

\begin{proof}
Using \eqref{ddcRho} and \eqref{karius}, we get 
\begin{multline*}
\lim_{j\to \infty} \dbar \chi_{\epsilon_j} \wedge \frac{1}{2\pi i}\left( c\frac{\partial
    f}{f}+\partial v\right) \wedge T= 
\lim_{j\to \infty} dd^c (\rho_j \circ \varphi) \wedge T -\lim_{j\to \infty} \chi_{\epsilon_j}
dd^c v \wedge T = \\
dd^c \varphi \wedge \1_{X\setminus Z} T - dd^c v\wedge\1_{X\minus Z} T. 
\end{multline*}
Since \eqref{hemmavid} is independent of the decomposition
$\varphi=u+a$, 
it follows that  $dd^c \ph \wedge \1_{X\setminus Z} T - dd^c v \wedge
     \1_{X\setminus Z} T = dd^c (c\log |f|^2) \wedge \1_{X\setminus Z}
     T$, and thus \eqref{eq:partial-ddcRho} follows. 
\end{proof}

\begin{lma}\label{socker}
    Assume that $\theta$ and $\eta$ are 
    closed $(1,1)$-forms and let 
\begin{equation*}
\alpha_j =
\big(\eta+\chi_{\epsilon_j}\cdot(\theta-\eta)+dd^c(\rho_{j}\circ\ph)\big)^{m}. 
\end{equation*} 
    Then there exist smooth forms $\Theta_{\ell,1}$ and $\Theta_{\ell,2}$, $\ell=1,\dotsc,m$, independent of $j$, such that
    \begin{equation*}
      \alpha_j
        =\eta^m+\sum_{\ell=1}^m \left(\frac{\dbar \chi_{\epsilon_j}^\ell}{f} \wedge \Theta_{\ell,1} + \frac{\chi_{\epsilon_j}^\ell}{f}\cdot \Theta_{\ell,2}  \right). 
    \end{equation*}
   
 Furthermore, if $T$ is a current with analytic singularities, then
    \begin{equation*}
        \lim_{j\rightarrow \infty} \alpha_j \wedge T    
        =[\theta +dd^c\varphi]^{m}_{\eta}\wedge T. 
    \end{equation*}
\end{lma}

\begin{proof}
First note that 
         \begin{equation*}
        \alpha_j 
        = \sum_{\ell=0}^m \binom m \ell \eta^{m-\ell}\wedge
        \big(\chi_{\epsilon_j}\cdot(\theta-\eta)+dd^c(\rho_j\circ\ph)\big)^\ell. 
        \end{equation*}
Set $\beta=\theta-\eta+dd^c v$. Then by 
  \eqref{ddcRho} 
\begin{equation*}
 \chi_{\epsilon_j}\cdot(\theta-\eta)+dd^c(\rho_j\circ\ph)
        =\chi_{\epsilon_j}\beta+ \dbar \chi_{\epsilon_j}\wedge\frac{1}{2\pi i} \left( c\frac{\partial f}{f}+\partial v\right)
    \end{equation*}
and using that $(\dbar \chi_{\epsilon_j})^2=0$ and
$\ell\chi_{\epsilon_j}^{\ell-1}\dbar
\chi_{\epsilon_j}=\dbar\chi_{\epsilon_j}^{\ell}$ we get 
\begin{equation*}
       \big ( \chi_{\epsilon_j}\cdot(\theta-\eta)+dd^c(\rho_j\circ\ph)
       \big )^\ell
=     \left(\chi_{\epsilon_j}^\ell\beta+ \dbar
  \chi_{\epsilon_j}^\ell\wedge\frac{1}{2\pi i} \left( c\frac{\partial
      f}{f}+\partial v\right)\right) \wedge \beta^{\ell-1}. 
      \end{equation*}

Thus 
\begin{equation}\label{eq:Yamuna} 
        \alpha_j 
        = \eta^m + \sum_{\ell=1}^m \binom m \ell \eta^{m-\ell}\wedge
\left(\chi_{\epsilon_j}^\ell\beta+ \dbar
  \chi_{\epsilon_j}^\ell\wedge\frac{1}{2\pi i} \left( c\frac{\partial
      f}{f}+\partial v\right)\right) \wedge \beta^{\ell-1},  
        \end{equation}
so that $\alpha_j$ is of the desired form with 
$\Theta_{\ell,1}:=\binom m \ell \eta^{m-\ell}\wedge \beta^{\ell-1} \wedge
\frac{1}{2\pi i}(c{\partial f}+f\partial v)$ and $\Theta_{\ell,2}:=\binom
m \ell \eta^{m-\ell}\wedge \beta^{\ell}\cdot f.$

\smallskip 
  
Next, by \eqref{PMproducts} and Lemma ~\ref{lma:residueVsMA}, since
$\chi^\ell$ is a cut-off function whenever $\chi$ is, 
 \begin{equation}\label{eq:Sava}
        \lim_{j\rightarrow \infty}
        \left(\chi_{\epsilon_j}^\ell\beta+ \dbar \chi_{\epsilon_j}^\ell\wedge \frac{1}{2\pi i} \left( c\frac{\partial f}{f}+\partial v\right)\right)
        \wedge T
        = (\beta + c\,dd^c\log|f|^2)\wedge\1_{X\setminus Z} T.
    \end{equation}
Since $\1_{X\setminus Z} dd^c\log |f|^2 \wedge T=0$ and
$\beta$ is smooth, 
by induction over $\ell$ we get that 
    \begin{equation}\label{eq:Inn}
   \big((\beta + c\,dd^c\log|f|^2)\wedge\1_{X\setminus Z}\big)^\ell\wedge
   T =     (\beta + c\,dd^c\log|f|^2)\wedge\1_{X\setminus Z}\beta^{\ell-1} \wedge T.
    \end{equation}
    Combining \eqref{eq:Yamuna}, \eqref{eq:Sava} and \eqref{eq:Inn} we
    get 
    \begin{multline*}
        \lim_{j\rightarrow \infty}\alpha_j\wedge T 
        =
\eta^m\wedge T+\sum_{\ell=1}^m \binom m \ell \eta^{m-\ell}\wedge (\beta+c\,dd^c \log |f|^2)\wedge   \1_{X\setminus Z} \beta^{\ell-1} T
=\\
\eta^m\wedge T+\sum_{\ell=1}^m \binom m \ell \eta^{m-\ell}\wedge \big
        ((\beta+c\,dd^c \log |f|^2)\wedge\1_{X\setminus Z} \big )^\ell T
=\\ 
\big ( \eta + (\beta+c\,dd^c \log |f|^2)\wedge\1_{X\setminus Z} \big
)^m \wedge T
        =\big(\eta\1_Z+(\theta+dd^c\ph)\1_{X\setminus Z}\big)^{m} \wedge T
        = \\ [\theta+dd^c\ph]_{\eta}^m\wedge T. 
    \end{multline*}	
\end{proof}

To prove Theorem ~\ref{paronmos} we need the following more general
version of \eqref{eq:CH2}, which essentially follows from the proof of
Theorem~11 in \cite{LS}.

\begin{prop} \label{LSvariant}
    Let $P_k^{\epsilon}$ be either $\tilde\chi_{k,\epsilon}/f_k$ or $\dbar\tilde\chi_{k,\epsilon}/f_k$, where $\epsilon > 0$, 
    $f_k$ is a holomorphic function and $\tilde\chi_{k,\epsilon} = \chi_k(|f_k|^{2c_k}e^{v_k}/\epsilon)$, where $\chi_k$ is a cut-off function, $c_k > 0$, and 
    $v_k$ is smooth, for $k=1,\dots,r$. For any $(\epsilon_1,\dots,\epsilon_r)$ that tends to $0$ along an admissible path,
    \begin{equation*}
\lim_{\nu\to\infty} P_r^{\epsilon_r(\nu)} \wedge \dots \wedge
P_1^{\epsilon_1(\nu)} = 
        \lim_{\epsilon_r'\to 0} \cdots \lim_{\epsilon_1' \to 0}
        P_r^{\epsilon_r'} \wedge \dots \wedge P_1^{\epsilon_1'}. 
    \end{equation*}
\end{prop}

Note that the difference between $\tilde\chi_{k,\epsilon}$ and
$\chi_{k,\epsilon}$ in \eqref{boman} is that we allow different cut-off functions
$\chi_k$ in $\tilde\chi_{k,\epsilon}$. 

For this result, it is crucial that the $v_k$ are smooth. Indeed, the proof below uses a change of variables involving $e^{v_k}$,
and this would not be possible if $v_k$ was just assumed to be a locally
bounded psh function.

\begin{proof}
If $\tilde \chi_{k,\epsilon} = \chi(|f_k|^2/\epsilon)$, i.e., when
$\chi_k=\chi$ for some cut-off function $\chi$, $c_k=1$ and $v_k=0$, then this indeed follows from \cite{LS}*{Theorem~11}.
To reduce to the case $c_k=1$, one lets $\hat{\chi}_k(t) = \chi_k(t^{c_k})$, which is also a cut-off function, $\hat{v}_k = v_k/c_k$ and $\hat{\epsilon}_k = \epsilon_k^{1/c_k}$, so that $\tilde\chi_{k,\epsilon_k} = \hat{\chi}_k(|f|^2e^{\hat{v}_k}/\hat{\epsilon}_k)$.
To allow for general $v_k$ and $\chi_k$, one just has to observe that the proof goes through in the same way in that situation. Indeed, 
to allow for the case that $v_k \not\equiv 0$, one just notices that in the beginning of the proof of \cite{LS}*{Theorem~11}, one may simply replace $\xi$
by $\xi$ times (the pullback to a resolution of singularities of) $e^{v_k}$. 
To allow for different $\chi_k$, in the proof, where $\chi_j^\epsilon = \chi(|x^{\tilde{\alpha}_j}|^2\xi_j/\epsilon_{\nu(j)})$
or $\chi_j^\epsilon = \chi(|y^{\tilde{\alpha}_j}|^2/\epsilon_{\nu(j)})$ appears, one just replaces $\chi$ in the right-hand side by $\chi_j$ and the proof will proceed in exactly the same way.
\end{proof}

\begin{proof}[Proof of Theorem ~\ref{paronmos}]
As above, let $\chi=\rho'\circ \log$, cf.\ Remark \ref{delsjon}. 
Since \eqref{eq:Amazon} is a local statement, we may assume that each
$\varphi_k$ is of the form \eqref{allihopa}. Moreover, after a
principalization and a resolution of singularities we may assume that
each $f_k$ is a single holomorphic function, cf.\
\cite{ABW}*{Section~4} and Remark ~\ref{hund}. 
By recursively applying the second part of Lemma ~\ref{socker} we have 
	\begin{equation*}
	\lim_{j_r'\rightarrow\infty}\cdots\lim_{j_1'\rightarrow\infty}\alpha^{(r)}_{j_r'}\wedge \dots \wedge\alpha^{(1)}_{j_1'}
	=[\theta_r + dd^c\varphi_r]^{m_r}_{\eta_r}\wedge \cdots \wedge
        [\theta_1 +dd^c\varphi_1]^{m_1}_{\eta_1}  
	\end{equation*}	
and thus it suffices to prove 
\begin{equation}\label{halsen}
	\lim_{\nu\rightarrow\infty}\alpha^{(r)}_{j_r(\nu)}\wedge \dots \wedge\alpha^{(1)}_{j_1(\nu)}
	=\lim_{j_r'\rightarrow\infty}\cdots\lim_{j_1'\rightarrow\infty}\alpha^{(r)}_{j_r'}\wedge
        \dots \wedge\alpha^{(1)}_{j_1'}. 
	\end{equation}

As above let $\epsilon_j=e^{-j}$, and let 
	\begin{equation*}
	P_{k,\ell,1}^j={\dbar (\chi_{k,\epsilon_j})^\ell}/{f_k}\hbox{ and\ }P_{k,\ell,2}^j={(\chi_{k,\epsilon_j})^\ell}/{f_k}
	\end{equation*}
for $k=1,\dotsc,r$, $\ell=1,\dotsc,m_k$. 
Since $\chi^\ell$ is a cut-off function
whenever $\chi$ is, it follows that $P^j_{k,\ell,i}$ are as in
Proposition ~\ref{LSvariant}. 
By  the first part of Lemma ~\ref{socker} there exist smooth forms $\Theta_{K,L,I}$ such that 
	\begin{equation}\label{eq:Madeira}
	\alpha^{(r)}_{j_r}\wedge \dots \wedge\alpha^{(1)}_{j_1}=
	\eta_r^{m_r}\wedge\dots\wedge\eta_1^{m_1}+ 
    \sum_{s=1}^r \sum_{K,L,I} \Theta_{K,L,I}\wedge P_{k_s,\ell_s,i_s}^{j_{k_s}}\wedge \dots \wedge P_{k_1,\ell_1,i_1}^{j_{k_1}}
	\end{equation}
where the inner sum is taken over all integer tuples $K=(k_1,\dotsc,k_s)$ with $1\leq k_1<\dots<k_s\leq r$,
all integer tuples $L=(\ell_1,\dotsc,\ell_s)$ with $1\leq \ell_\kappa\leq m_{k_\kappa}$, $\kappa=1,\dotsc, s$,
and all tuples $I=(i_1,\dots,i_s)$ with $i_\kappa \in \{1,2\}$, $\kappa=1,\dots,s$. 

Since $(j_1,\dotsc,j_r)\colon \N\rightarrow \R^r$ tends
to $\infty$ along an admissible path, then so does
$(j_{k_1},\dotsc,j_{k_s})\colon \N\rightarrow \R^s$, if
$K=(k_1,\dotsc,k_s)$ is as above, and so $(\epsilon_{k_1},\ldots,
\epsilon_{k_s})$ tends to $0$ along an admissible path.
Thus, by Proposition ~\ref{LSvariant} 
	\[
	\lim_{\nu\rightarrow\infty} P_{k_s,\ell_s,i_s}^{j_{k_s}(\nu)}\wedge \dots \wedge P_{k_1,\ell_1,i_1}^{j_{k_1}(\nu)} =
	\lim_{j_s'\rightarrow\infty}\cdots\lim_{j_1'\rightarrow\infty}
	P_{k_s,\ell_s,i_s}^{j_s'}\wedge \dots \wedge
        P_{k_1,\ell_1,i_1}^{j_1'}, 
	\]
and hence \eqref{halsen} follows in view of \eqref{eq:Madeira}. 
\end{proof}

\begin{remark}
With simple adaptations to the above proof, we get regularizations
also of the more general mixed Monge-Amp\`ere products \eqref{katta}. 
For instance, let $\ph_1, \ph_2, \ph_3$  be qpsh functions with analytic singularities, 
let $Z_2$ be the unbounded locus of $\ph_2$, and
let $(j_1,j_2,j_3)\colon \N\rightarrow \R^3$ be a sequence tending to $\infty$ along an admissible path. 
Then,
\begin{multline*}
    \lim_{\nu\rightarrow\infty}
    \big(dd^c (\rho_{j_{3}(\nu)} \circ \varphi_3)\big)^{m_3} \wedge(\rho_{j_2(\nu)}'\circ \ph_{2})\cdot\big(dd^c  (\rho_{j_1(\nu)}\circ \varphi_1)\big)^{m_1}\\
    = (dd^c \varphi_3)^{m_3}\wedge\1_{X \setminus Z_2} (dd^c \varphi_1)^{m_1}. 
\end{multline*}
\end{remark}

\begin{remark}
\label{rmk:Lucas}
It could appear natural in the situation of Theorem~\ref{huvud} to
consider one parameter limits like
\begin{equation} \label{badlimits}
        \lim_{j\rightarrow \infty}
        \big(dd^c(\rho_j\circ\ph_r)\big)^{m_r} \wedge\dots\wedge   \big(dd^c(\rho_j \circ\ph_1)\big)^{m_1}, 
\end{equation}
i.e., where all the $j_k$ are all equal to a single $j$. This would correspond to letting all the $\epsilon_k$ in 
Proposition~\ref{LSvariant} be equal to a single $\epsilon$. If $P_k^\epsilon$ are as in Proposition~\ref{LSvariant}, then
limits of expressions like 
 $P_r^{\epsilon_r} \wedge \dots \wedge P_1^{\epsilon_1}$
are very sensitive to how $(\epsilon_1,\dots,\epsilon_r)$ tends to $0$.
In fact, if we let
\begin{equation*}
    I(\mathbf{s}) := \lim_{\delta\to 0} P_r^{\delta^{s_r}} \wedge \dots \wedge P_1^{\delta^{s_1}},
\end{equation*}
where $\mathbf{s}=(s_1,\dots,s_r) \in \R_{> 0}^r$, 
then by \cite{P}*{Proposition~1}, there exist finitely many vectors $\mathbf{n}_i \in \mathbf{Q}^r$, $i=1,\dots,N,$ such that
$I(\mathbf{s})$ is well-defined and locally constant on $\R_{> 0}^r \setminus \cup \{ \mathbf{n}_i \cdot \mathbf{s} = 0 \}$. 
The case above with all $\epsilon_k$
equal to $\epsilon$ corresponds to when $\mathbf{s}=(1,\dots,1)$, and it could very well happen that $\mathbf{s}$ lies in one
of the hyperplanes $\{\mathbf{n}_i \cdot \mathbf{s} = 0 \}$, in which
case we would not know whether $I(1,\dots,1)$ is well-defined. Hence,
we do not know in general if the limit \eqref{badlimits} exists. 
\end{remark}

\begin{remark}
    Assume that we are in the situation of Theorem~\ref{huvud}, or
    more generally Theorem~\ref{paronmos}.
    By choosing $1<r_1<r_2<\dots<r_p=r$, we may divide $\{1,\dots,r\}$ into $p$
    blocks 
    \begin{equation*}
            \{1,\dots,r_1\},\{r_1+1,\dots,r_2\},\dots,\{r_{p-1}+1,\dots,r_p\}.
    \end{equation*}
    It could be natural to consider limits that tend to $\infty$ along admissible paths iteratively
    in each block, so that the left-hand side in Theorem~\ref{huvud} corresponds to the iterated limit
    when there is just a single block $\{1,\dots,r\}$, while the right-hand side corresponds to
    the limit when we have $r$ blocks $\{1\},\dots,\{r\}$.

In fact, in \cite{LS} certain generalized admissible paths are
considered that give regularization results like this for residue
currents. By small adaptations of our proofs to this situation we
would get results like 
   \begin{multline*}
    \lim_{\nu_p\to\infty}\cdots \lim_{\nu_1\to\infty} 
    \big(dd^c(\rho_{{j_{r_p}}(\nu_p)}\circ\ph_{r_p})\big)^{m_{r_p}} \wedge
    \cdots \wedge
    \big(dd^c(\rho_{j_{r_{p-1}+1}(\nu_p)}\circ\ph_{r_{p-1}+1})\big)^{m_{r_{p-1}+1}}
    \wedge \\
    \cdots \wedge 
    \big(dd^c(\rho_{j_{r_1}(\nu_1)}\circ\ph_{r_1})\big)^{m_{r_1}}
    \wedge\cdots \wedge  
\big(dd^c(\rho_{{j_1}(\nu_1)}\circ\ph_1)\big)^{m_1} \Big)
    = 
    (dd^c \varphi_r)^{m_r}\wedge\cdots\wedge (dd^c \varphi_1)^{m_1}
    . 
    \end{multline*}
if each $(j_{r_k+1}, \ldots, j_{r_{k+1}})$ tends to infinity along an
admissible path. 
\end{remark}

\section{Chern and Segre forms of metrics with analytic singularities}
\label{sec:Segre}

In \cite{LRSW}, we use generalized mixed Monge-Amp\`ere products to construct
Chern and Segre forms, or rather currents, for hermitian metrics on
holomorphic vector bundles that have analytic singularities in a
certain sense. 
In this section, we apply the results presented above to get an
approximation of these Chern and Segre currents by smooth forms in
the corresponding Chern and Segre classes. 

Let us briefly recall the construction in \cite{LRSW}; for details and
references we refer to that paper. 
Assume that $E\to X$ is a holomorphic vector bundle of rank $r$. 
Let us first consider the classical
setting and assume that $h$ is a smooth hermitian metric on $E$. 
Let $\pi : \P(E)\to X$ be the
projective bundle of lines in $E^*$. Then $h^*$ induces a metric on the
tautological line bundle $\Ok_{\P(E)}(-1)\subset \pi^* E^*$;  let
$e^{-\phi}$ be the dual metric on $L:=\Ok_{\P(E)}(1)$. 
If $h$ is Griffiths semipositive, then
$e^{-\phi}$ is a semipositive metric, i.e., the local weights $\phi$
are psh. 
The $k$th \emph{Segre form} can be defined as 
\begin{equation}\label{ny}
s_k(E,h):=(-1)^k\pi_*(dd^c\phi)^{k+r-1}.
\end{equation}
This definition coincides with the classical 
definition of Segre forms, which means that the total Segre form
$s(E,h)=1+s_1(E,h)+s_2(E,h)+\cdots$ is the multiplicative
inverse of the total
Chern form $c(E,h)=1+c_1(E,h)+c_2(E,h)+\cdots$. 

In \cite{LRSW} we considered Griffiths semipositive singular metrics $h$ on $E$ in the sense of
Berndtsson-P\u aun, \cite{BP}, such that the corresponding singular
metrics $e^{-\phi}$ on $L$ satisfy that the local weights $\phi$ are psh with analytic
singularities\footnote{Recall that in \cite{LRSW} we use the less
  restrictive definition of analytic singularities, cf.\ Remark ~\ref{rem:nonSmooth} above.}; we say that such $h$ have \emph{analytic
  singularities}. For these metrics we constructed Chern and Segre forms
by mimicking the smooth setting. 
Let $\theta$ be a first Chern form of a smooth metric $e^{-\psi}$ on $L$,
and let 
\begin{equation}\label{kanal}
s_k(E,h,\theta):=(-1)^k\pi_\ast [dd^c\phi]_\theta^{k+r-1}, 
\end{equation}
see Remark ~\ref{kaktus}; this is a closed normal $(k,k)$-current. 
Since $[dd^c\phi]_\theta^{m}=(dd^c\phi)^m$ where $h$ is smooth, cf.\
Remark ~\ref{sommaren}, it follows that $s_k(E,h,\theta)$ coincides
with the classical Segre form $s_k(E,h)$ where $h$ is
smooth. Moreover, by Proposition ~\ref{vinter}, 
$[dd^c\phi]_\theta^{m}$ is cohomologous to $\theta^m$, and thus $s_k(E,h,\theta)$ is in
the $k$th Segre class $s_k(E)$ of $E$, i.e., the class of the $k$th Segre
form of a smooth metric.

To construct Chern forms we defined products of the Segre forms
\eqref{kanal}. 
Let $E_1,\ldots, E_t$ be $t$ 
disjoint copies of $E$ and let $\varpi: Y\to X$ be the fiber product
$Y=\P(E_t)\times_X\cdots \times_X \P(E_1).$ Let $\phi_i$ and
$\theta_i$ denote the pullbacks to $Y$ of the metric and form on
$\P(E_i)$ corresponding to $\phi$ and $\theta$,
respectively. 
Now, for $k_1,...,k_t\geq 1$, we define 
\begin{equation*}
s_{k_t}(E,h,\theta)\wedge\cdots\wedge s_{k_1}(E,h,\theta)
:=
(-1)^k\varpi_* \big ( [dd^c
\phi_t]_{\theta_t}^{k_t+r-1}\wedge\cdots\wedge [dd^c
\phi_1]_{\theta_1}^{k_1+r-1}
\big ), 
\end{equation*}
where $k= k_1+\cdots + k_t$,  see Remark ~\ref{kaktus}, 
and  
\begin{equation} \label{eq:Chern}
 c_k(E,h,\theta) := \sum_{k_1+\cdots
   +k_t=k}(-1)^ts_{k_t}(E,h,\theta)\wedge\cdots\wedge
 s_{k_1}(E,h,\theta), 
\end{equation}
so that the total Chern form $1+c_1(E,h,\theta) + \cdots$ times the
total Segre form $1 + s_1(E, h, \theta)
+ \cdots$ equals $1$. 
As above, it follows from the construction that $c_k(E, h, \theta)$ coincides with the classical Chern form $c_k(E,
h)$ where $h$ is smooth and that it is in the $k$th Chern class $c_k(E)$ of
$E$, see \cite{LRSW}*{Theorem~1.1}. 
We also show that
$s_k(E, h,\theta)$ and $c_k(E,h,\theta)$ coincide with the Chern and Segre forms
for singular metrics defined by the first two authors and Raufi and
Ruppenthal in \cite{LRRS} when these are defined. Moreover, we show
that although the currents $s_k(E, h,
\theta)$ and $c_k(E,h,\theta)$ depend on the choice of $\theta$ in
general, the Lelong numbers at each point $x\in X$ are independent of
$\theta$.

\medskip

We want to use our regularization results 
to regularize these currents. 
Let $\rho_j$ be as in Definition~\ref{df:rho} and let 
\begin{equation*}
\alpha_{k,j}=\big (\theta+dd^c (\rho_j \circ \varphi)\big )^{k+r-1} 
\quad \text{ and } \quad 
\beta_{k,j} = (-1)^k \pi_* \alpha_{k,j}, 
\end{equation*}
where $\varphi$ is the qpsh function $\varphi=\phi-\psi$, cf. \eqref{herde} and Remark \ref{kaktus}.  
Then $\beta_{k,j}$ is a smooth form since it is the direct image of a
smooth form under a submersion. 
Moreover, clearly $\alpha_{k,j}$ is cohomologous to $\theta^{k+r-1}$
and thus $\beta_{k,j}\in s_k(E)$,
cf.\ \eqref{ny}. 
From Theorem~\ref{paron} we get the following regularization result. 
\begin{cor}
\label{cor:application}
Assume that we are in the situation above. 
If $(j_1,\dots,j_t) : \N \to \R^t$ tends to $\infty$ along an admissible path, then
	\begin{equation*}
\lim_{\nu \to \infty} \beta_{k_t,j_t(\nu)}\wedge \dots \wedge
\beta_{k_1,j_1(\nu)} = 
	s_{k_t}(E,h,\theta)\wedge\cdots\wedge s_{k_1}(E,h,\theta). 
	\end{equation*}
\end{cor}

In particular, in view of \eqref{eq:Chern}, it follows that $s_k(E,h,\theta)$ and $c_k(E,h,\theta)$ are given as
limits of smooth forms in the classes $s_{k}(E)$ and $c_k(E)$, respectively.  

\begin{proof}
Following the notation in \cite{LRSW}, let 
$\tilde\alpha_{k,j,i}$ be the form on $\P(E_i)$ corresponding to $\alpha_{k,j}$.
Moreover let $\alpha_{k,j,i}=\varpi_i^*\tilde \alpha_{k,j,i}$,
where $\varpi_i$ is the projection $Y\to \mathbf P(E_i)$.
By Theorem ~\ref{paron}, in view of \eqref{lakrits},  
\begin{multline}\label{forsta}
	[dd^c\phi_t]_{\theta_t}^{k_t+r-1}\wedge\cdots\wedge
        [dd^c\phi_1]_{\theta_1}^{k_1+r-1} 
= \lim_{\nu\to\infty} \alpha_{k_t,j_t(\nu),t}\wedge\cdots\wedge
\alpha_{k_1,j_1(\nu),1}
=\\
= \lim_{\nu\to\infty} \varpi_t^*\tilde\alpha_{k_t,j_t(\nu),t}\wedge\cdots\wedge
\varpi_1^* \alpha_{k_1,j_1(\nu),1}. 
    \end{multline}
Let $\pi_i$ be the projection $\P(E_i)\to X$. 
Applying $\varpi_*$ to \eqref{forsta}, using that $\beta_{k,j} = (-1)^k(\pi_i)_*\tilde\alpha_{k,j,i}$ and
that 
\[
\varpi_* (\varpi_t^*\gamma_t\wedge\cdots\wedge \varpi_1^* \gamma_1)=
(\pi_t)_*\gamma_t\wedge\cdots\wedge(\pi_1)_* \gamma_1
\]
for all smooth forms $\gamma_1,\ldots, \gamma_t$ on $\P(E_1), \ldots,
\P(E_t)$, respectively, see, e.g., \cite{LRSW}*{Lemma 6.3}, 
we obtain 
\begin{multline*}
s_{k_t}(E,h,\theta)\wedge\cdots\wedge s_{k_1}(E,h,\theta)
=
\lim_{\nu\to\infty} \varpi_* \big (\varpi_t^*\tilde\alpha_{k_t,j_t(\nu),t}\wedge\cdots\wedge
\varpi_1^* \tilde\alpha_{k_1,j_1(\nu),1} \big ) 
=\\
\lim_{\nu\to\infty} (\pi_t)_*\tilde\alpha_{k_t,j_t(\nu),t}\wedge\cdots\wedge
(\pi_1)_* \tilde\alpha_{k_1,j_1(\nu),1} 
=
\lim_{\nu \to \infty} \beta_{k_t,j_t(\nu)}\wedge \dots \wedge
\beta_{k_1,j_1(\nu)}.
\end{multline*}
\end{proof}

\end{document}